\documentclass[a4paper,12pt,reqno]{amsart}
\usepackage{a4wide}
\usepackage{amsmath}
\usepackage{amssymb}
\usepackage{amsthm}
\usepackage{latexsym}
\usepackage{graphicx}
\usepackage[english]{babel}
              
\newtheorem{prop}[subsection]{Proposition}

\newtheorem{teor}[subsection]{Theorem}
\newtheorem{lema}[subsection]{Lemma}
\newtheorem{cor} [subsection]{Corollary}
\theoremstyle{definition}

\theoremstyle{remark}

\newtheorem{exm} [subsection]{Example}

\newcommand{\npd}{\left\lfloor \frac{n}{2} \right\rfloor}

\def\dt{\operatorname{dt}}

\numberwithin{equation}{section}

\begin{document}

\title[A combinatorial approach to the Fourier expansions of powers of cos and sin]
      {A combinatorial approach to the Fourier expansions of powers of cos and sin}
\author[Mircea Cimpoea\c s]{Mircea Cimpoea\c s$^1$}  
\date{}

\keywords{Combinatorial identities; Trigonometric functions; Fourier series.}

\subjclass[2020]{05A19, 26A09, 42A20}

\footnotetext[1]{ \emph{Mircea Cimpoea\c s}, National University of Science and Technology Politehnica Bucharest, Faculty of
Applied Sciences, 
Bucharest, 060042, Romania and Simion Stoilow Institute of Mathematics, Research unit 5, P.O.Box 1-764,
Bucharest 014700, Romania, E-mail: mircea.cimpoeas@upb.ro,\;mircea.cimpoeas@imar.ro}

\begin{abstract}
We present a new combinatorial approach to the computation of the (real) Fourier expansions of $\cos^n(t)$ and $\sin^n(t)$,
where $n\geq 1$ is an integer. As an application, we compute the Fourier expansions of $f(t)=\frac{1}{a-\cos t}$ and
$g(t)=\frac{1}{a-\sin t}$, where $a\in\mathbb R$ with $|a|>1$.
\end{abstract}

\maketitle

\section{Introduction}

Fourier series plays a very important role, both in pure mathematics, but also in different branches of applied mathematics 
and engineering, like,  for instance, in the study of periodic signals. For a friendly introduction 
in the topic of Fourier series and Fourier analysis, we refer the reader to \cite{stein}.

It is well known that if $f:\mathbb R\to\mathbb R$ is a twice continuously differentiable 
periodic function, with the period $T=2\pi$, then $f(t)$ has the (real) Fourier expansion:
\begin{align*}
& f(t) = \frac{a_0}{2}+\sum_{n=1}^{\infty} a_n\cos(nt) + \sum_{n=1}^{\infty} b_n\sin(nt),\text{ where }\\
& a_n=\frac{1}{\pi}\int_{0}^{2\pi} f(t)\cos(nt)\dt,\text{ for all }n\geq 0,\text{ and }
  b_n=\frac{1}{\pi}\int_{0}^{2\pi} f(t)\sin(nt)\dt,\text{ for all }n\geq 1.
\end{align*}	
Moreover, the Fourier series is uniformly convergent; see \cite[Corollary 2.2.4]{stein}.

Our aim is to find a new, combinatorial, approach to the computation of the the 
Fourier series of $\cos^n t$ and $\sin^n t$. Our paper is structured as follows:

In Section $2$ we prove
several combinatorial formulas which will be used later on. For instance, in Lemma \ref{lem1}, we prove that \small
$$  \sum_{j\geq 0} \binom{n}{2j}\binom{j}{s} = 2^{n-2s-1}\left(\binom{n-s}{s}+\binom{n-s-1}{s-1}\right),
    \sum_{j\geq 0} \binom{n}{2j+1}\binom{j}{s}  =2^{n-2s-1}\binom{n-s-1}{s},$$
    \normalsize $\text{ for all }n\geq 1,s\geq 0$.
		
Section $3$ is dedicated to our main results. In Proposition \ref{p1} we proved that for any integer $n\geq 1$, we have that
$$\cos(nt) = \sum_{s=0}^{\npd} (-1)^s 2^{n-2s-1}\left(\binom{n-s}{s}+\binom{n-s-1}{s-1}\right) \cdot \cos^{n-2s}(t),\text{ for all }t\in\mathbb R.$$
In Proposition \ref{p2}, we prove similar formulas for $\sin(nt)$, where $n\geq 1$.

In Theorem \ref{t1} we proved that for any integer $k\geq 0$, we have that
\begin{align*}
& \cos^{2k}(t) = \frac{1}{2^{2k-1}} \sum_{j=0}^{k-1} \binom{2k}{j} \cos((2k-2j)t) + \frac{1}{2^{2k}}\binom{2k}{k}\text{ and }\\
& \cos^{2k+1}(t) = \frac{1}{2^{2k}} \sum_{j=0}^{k} \binom{2k+1}{j} \cos((2k+1-2j)t),\text{ for all }t\in\mathbb R.
\end{align*}
Using the fact that $\sin(t)=\cos(t-\frac{\pi}{2})$, we derive similar formulas for $\sin^{2k}(t)$ and $\sin^{2k+1}(t)$;
see Corollary \ref{c1}. 

As an application, in Section $4$, 
we compute the Fourier expansions of $f(t)=\frac{1}{a-\cos t}$ and $g(t)=\frac{1}{a-\sin t}$, where $a\in\mathbb R$
with $|a|>1$; see Theorem \ref{t2} and Corollary \ref{c2}.

\section{Combinatorial lemmata}

For any integers $n\geq 0$ and $s\geq 0$, we denote
\begin{equation}\label{alfa}
\alpha_{n,s} = \sum_{j\geq 0} \binom{n}{2j}\binom{j}{s},\;\alpha'_{n,s} = \sum_{j\geq 0} \binom{n}{2j-1}\binom{j}{s}
\text{ and }\alpha''_{n,s}=\sum_{j\geq 0} \binom{n}{2j+1}\binom{j}{s}.
\end{equation}
Note that $\alpha_{n,s}=0$, for $s>\npd$, $\alpha'_{n,s}=0$, for $s>\left\lfloor \frac{n+1}{2} \right\rfloor$
and $\alpha''_{n,s}=0$, for $s>\left\lfloor \frac{n-1}{2} \right\rfloor$.

\begin{lema}\label{lem1}
With the above notations, we have:
\begin{align*}
&(1)\; \alpha_{n,s} = 2^{n-2s-1}\left(\binom{n-s}{s}+\binom{n-s-1}{s-1}\right),\text{ for all }n\geq 1,s\geq 0, \\
&(2)\; \alpha'_{n,s} = 2^{n-2s-1}\left(\binom{n+1-s}{s}+2\binom{n-s}{s-1}+\binom{n-s-1}{s-2}\right),\text{ for all }n\geq 1,s\geq 0,\text{ and},\\
&(3)\; \alpha''_{n,s}=2^{n-2s-1}\binom{n-s-1}{s}\text{ for all }n\geq 1,s\geq 0.
\end{align*}
\end{lema}

\begin{proof}
First, note that from \eqref{alfa}, we have $\alpha_{1,0}=1$ and $\alpha_{1,s}=0$ for $s\geq 1$.
Also, we have $\alpha'_{1,0}=\alpha'_{1,1}=1$ and $\alpha'_{1,s}=0$ for $s\geq 1$. It follows that the required
formulas hold for $n=1$ and $s\geq 0$. On the other hand, since $\sum_{j=0}^n \binom{n}{2j} = \sum_{j=0}^n \binom{n}{2j-1} = 2^{n-1}$,
it follows that $\alpha_{n,0}=\alpha'_{n,0}=2^{n-1}$ and, thus, the required formulas hold for $n\geq 1$ and $s=0$.

Let $n\geq 1$. From \eqref{alfa} it follows that
\begin{equation}\label{rec1}
\alpha_{n+1,s}=\sum_{j\geq 0} \binom{n+1}{2j}\binom{j}{s} = \sum_{j\geq 0} \binom{n}{2j}\binom{j}{s}+
\sum_{j\geq 0}\binom{n}{2j-1}\binom{j}{s} = \alpha_{n,s}+\alpha'_{n,s}.
\end{equation}
Assume $s\geq 1$. From \eqref{alfa} it follows also that
\begin{equation}\label{rec2}
\alpha'_{n+1,s} 
= \sum_{j\geq 0}\binom{n}{2j-1}\binom{j}{s}
+ \sum_{j\geq 0}\binom{n}{2j-2}\binom{j}{s} = \alpha'_{n,s}+\sum_{j\geq 0}\binom{n}{2j}\binom{j+1}{s}.
\end{equation}
On the other hand, we have
\begin{equation}\label{rec3}
\sum_{j\geq 0}\binom{n}{2j}\binom{j+1}{s} = \sum_{j\geq 0}\binom{n}{2j}\binom{j}{s} + \sum_{j\geq 0}\binom{n}{2j}\binom{j+1}{s-1}
= \alpha_{n,s}+\alpha_{n,s-1}.
\end{equation}
From \eqref{rec1}, \eqref{rec2} and \eqref{rec3} it follows that $\alpha_{n,s}$'s and $\alpha'_{n,s}$'s satisfy the recurrence 
relations:
\begin{equation}\label{rec}
\begin{cases} \alpha_{n+1,s} = \alpha_{n,s} + \alpha'_{n,s},\text{ for all }n\geq 1,s\geq 0, \\ 
              \alpha'_{n+1,s} = \alpha'_{n,s} + \alpha_{n,s}+\alpha_{n,s-1}\text{ for all }n\geq 1,s\geq 1 \end{cases}.
\end{equation}
Hence, we can apply an inductive argument in order to complete the proof, as the required formulas hold for $n=1$ and $s\geq 0$, and,
also, for $n\geq 1$ and $s=0$. Assume $n\geq 1$ and $s\geq 1$ and the formulas hold for the pairs $(n,s)$ and $(n,s-1)$. 
From \eqref{rec} it follows that
$$\alpha_{n+1,s}=  2^{n-2s-1}\left( \binom{n-s}{s}+\binom{n-s-1}{s-1} + \binom{n+1-s}{s}+2\binom{n-s}{s-1}+\binom{n-s-1}{s-2} \right).$$
Since $\binom{n-s}{s}+\binom{n-s}{s-1}=\binom{n+1-s}{s}$ and $\binom{n-s-1}{s-1}+\binom{n-s-1}{s-2}=\binom{n-s}{s-1}$, it follows that
$$\alpha_{n+1,s} = 2^{n-2s-1}\left( 2\binom{n+1-s}{s} + 2\binom{n-s}{s-1} \right) = 2^{n+1-2s-1}\left( \binom{n+1-s}{s} + \binom{n-s}{s-1}
\right),$$
as required. Also, from \eqref{rec} it follows that \small
$$\alpha'_{n+1,s} = \alpha_{n+1,s}+\alpha_{n,s-1} = 2^{n-2s-1} \left( 2\binom{n+1-s}{s} + 2\binom{n-s}{s-1} + 4\binom{n+1-s}{s-1}
+ 4\binom{n-s}{s-2} \right).$$ \normalsize
Since $\binom{n+1-s}{s}+\binom{n+1-s}{s-1}=\binom{n+2-s}{s}$ and $\binom{n-s}{s-1}+\binom{n-s}{s-2}=\binom{n+1-s}{s-1}$, we therefore obtain
$$\alpha'_{n+1,s}=2^{n-2s}\left( \binom{n+2-s}{s} + 2\binom{n+1-s}{s-1} + \binom{n-s}{s-2} \right),$$
as required. Note that
\begin{equation}\label{xo1}
\alpha''_{0,s}=0,\text{ for all }s\geq 0,\text{and },\alpha''_{n,0}=\sum_{j\geq 1}\binom{n}{2j+1} = 2^{n-1},\text{ for all }n\geq 0.
\end{equation}
Assume $n,s\geq 1$. Since $\binom{n}{2j+1}=\binom{n-1}{2j}+\binom{n-1}{2j+1}$, from \eqref{alfa} if follows that
\begin{equation}\label{xo2}
\alpha''_{n,s}=\alpha''_{n-1,s}+\alpha_{n-1,s},\text{ where }\alpha_{0,s}=0.
\end{equation}
We claim that 
\begin{equation}\label{clem}
\alpha''_{n,s}=2^{n-2s-1}\binom{n-s-1}{s}\text{ for all }n,s\geq 0.
\end{equation}
From \eqref{xo1} it follows that \eqref{clem} holds for $n=0$ or $s=0$. Assume $n,s\geq 1$.
In order to prove \eqref{clem}, using \eqref{xo2}, it suffices to show that
$$\alpha_{n-1,s}=2^{n-2s-1}\binom{n-s-1}{s} - 2^{n-2s-2}\binom{n-s-2}{s},$$
which follows immediately from (1) and the identity $\binom{n-s-1}{s}-\binom{n-s-2}{s}=\binom{n-s-2}{s-1}$. Hence, the proof is complete.
\end{proof}


\begin{lema}\label{lem2}
For all integers $0\leq s\leq k$, we have that:
$$\sum_{j=1}^s \binom{2k}{2j-1}\binom{k-j}{s-j} = 2^{2s-1}\binom{k+s-1}{2s-1}.$$
\end{lema}

\begin{proof}
We denote $u=k-j$. It follows that
$$\sum_{j=1}^s \binom{2k}{2j-1}\binom{k-j}{s-j} = \sum_{u\geq 0} \binom{2k}{2u+1}\binom{u}{k-s}.$$
Hence, the conclusion follows from Lemma \ref{lem1}(3).
\end{proof}

\begin{lema}\label{cooc}
Let $\ell\geq 0$ be an integer. For any integer $t\geq 1$, we have that:
$$ \sum_{s=0}^t (-1)^s \left(\binom{2\ell+s}{s} + \binom{2\ell+s-1}{s-1} \right)\binom{2t+2\ell}{t-s} = 0.$$
\end{lema}

\begin{proof}
Using the identity $(-1)^k\binom{x}{k}=\binom{-x+k-1}{k}$ and the Chu-Vandermonde summation (see \cite{Rior}), we get:
\begin{equation}\label{chu1}
\sum_{s=0}^t (-1)^s \binom{2\ell+s}{s} \binom{2t+2\ell}{t-s} = 
\sum_{s=0}^t \binom{-2\ell-1}{s} \binom{2t+2\ell}{t-s} = \binom{2t-1}{t}.
\end{equation}
Similarly, by denoting $u=s-1$, we have:
\begin{equation}\label{chu2}
\sum_{s=0}^t (-1)^s \binom{2\ell+s-1}{s-1} \binom{2t+2\ell}{t-s} = -
\sum_{u=0}^{t-1} \binom{2\ell+u}{u} \binom{2t+2\ell}{t-1-u} = \binom{2t-1}{t-1}.
\end{equation}
Since $t\geq 1$, the conclusion follows from \eqref{chu1} and \eqref{chu2}.
\end{proof}

\begin{lema}\label{cheie}
For all integers $n,\ell\geq 0$, we have that:
$$ \sum_{k=0}^{\ell} \left( \binom{n+2k-1}{k} - \binom{n+2k-1}{k-1} \right)\binom{2(\ell-k)}{\ell-k} = \binom{n+2\ell}{\ell}.$$
\end{lema}

\begin{proof}
We use induction on $n,\ell\geq 0$. 

Assume $n=0$ and $\ell\geq 0$. Since $\binom{-1}{0}=1$, $\binom{-1}{-1}=0$ and $\binom{2k-1}{k}=\binom{2k-1}{k-1}$ for all $k\geq 1$, it follows that
$$ \sum_{k=0}^{\ell} \left( \binom{2k-1}{k} - \binom{2k-1}{k-1} \right)\binom{2(\ell-k)}{\ell-k} = \binom{2\ell}{\ell}, $$
as required. Assume $n\geq 1$ and $\ell=0$. Since $\binom{n}{0}=\binom{n-1}{0}=\binom{0}{0}=1$ and $\binom{n-1}{-1}=0$, the assertion is trivial.

Now, assume $n,\ell\geq 1$. Since $\binom{n+2k-1}{k}=\binom{n+2k-2}{k}+\binom{n+2k-1}{k-1}$ and 
$\binom{n+2k-1}{k-1}=\binom{n+2k-2}{k-1}+\binom{n+2k-1}{k-2}$, it follows that
\begin{align*}
 S_{n,\ell}:&=\sum_{k=0}^{\ell} \left( \binom{n+2k-1}{k} - \binom{n+2k-1}{k-1} \right)\binom{2(\ell-k)}{\ell-k} = \\
 & = \sum_{k=0}^{\ell} \left( \binom{n+2k-2}{k} - \binom{n+2k-2}{k-1} \right)\binom{2(\ell-k)}{\ell-k} + \\
 & +  \sum_{k=0}^{\ell} \left( \binom{n+2k-2}{k-1} - \binom{n+2k-2}{k-2} \right)\binom{2(\ell-k)}{\ell-k} = \\
 & = S_{n-1,\ell} + \sum_{k=1}^{\ell} \left( \binom{n+2k-2}{k-1} - \binom{n+2k-2}{k-2} \right)\binom{2(\ell-k)}{\ell-k}.
\end{align*}
Denoting $s=k-1$, it follows that
$$ S_{n,\ell} = S_{n-1,\ell} + \sum_{s=0}^{\ell-1} \left( \binom{n+2s}{s} - \binom{n+2s}{s-1} \right)\binom{2(\ell-1-s)}{\ell-1-s}
= S_{n-1,\ell}+S_{n+1,\ell-1}.$$
Using the induction hypothesis on $n$ and $\ell$, it follows that
$$ S_{n,\ell} = \binom{n-1+2\ell}{\ell}+\binom{n+1+2\ell-2}{\ell-1} = \binom{n+2\ell}{\ell},$$
as required. Hence, the proof is complete.
\end{proof}

\section{Main results}

\begin{prop}\label{p1}
Let $n\geq 1$ be an integer. We have that:
\begin{align*}
& \cos(nt) = \sum_{s=0}^{\npd} (-1)^s 2^{n-2s-1}\left(\binom{n-s}{s}+\binom{n-s-1}{s-1}\right) \cdot \cos^{n-2s}(t),\text{ for all }t\in\mathbb R.
\end{align*}
\end{prop}

\begin{proof}
We consider the Moivre identity
\begin{equation}\label{ec1}
e^{int}=\cos(nt)+i\sin(nt)=(\cos t+i\sin t)^n=(e^{it})^n.
\end{equation}
From \eqref{ec1} it follows that
\begin{equation}\label{ec2}
\cos(nt)=\sum_{j=0}^{\npd}(-1)^j \binom{n}{2j} \cos^{n-2j}(t)\sin^{2j}(t).
\end{equation}
Since $\sin^2(t)=1-\cos^2(t)$, from \eqref{ec2} it follows that
\begin{equation}\label{ec3}
\cos(nt)=\sum_{j=0}^{\npd}(-1)^j \binom{n}{2j} \cos^{n-2j}(t) \sum_{\ell=0}^j (-1)^{\ell}\binom{j}{\ell} \cos^{2\ell}(t).
\end{equation}
Denoting $s=j-\ell$ in \eqref{ec2}, we obtain
\begin{equation}\label{ec4}
\cos(nt)=\sum_{s=0}^{\npd}\left((-1)^s \sum_{j=s}^{\npd} \binom{n}{2j}\binom{j}{j-s}\right) \cos^{n-2s}(t).
\end{equation}
On the other hand, we have:
\begin{equation}\label{ec5}
\sum_{j=s}^{\npd} \binom{n}{2j}\binom{j}{j-s} = \sum_{j=0}^n \binom{n}{2j}\binom{j}{s}.
\end{equation}
Now, the conclusion follows from Lemma \ref{lem1}, \eqref{ec4} and \eqref{ec5}.
\end{proof}

\begin{exm}\rm
(1) 
    From Proposition \ref{p1} it follows that $\cos(2t) = 2\cos^2 t - 1.$		
		
(2) 
    From Proposition \ref{p1}, it follows that
    $\cos(4t)=8\cos^4(t)-8\cos^2(t)+1.$
\end{exm}

\begin{prop}\label{p2}
\begin{enumerate}
\item[(1)] Let $k\geq 0$ be an integer. Then:
           $$\sin((2k+1)t) = \sum_{s=0}^{k} (-4)^{k-s}\left(\binom{2k+1-s}{s}+\binom{2k-s}{s-1}\right) \cdot \sin^{2k+1-2s}(t),\text{ for all }t\in\mathbb R.$$
\item[(2)] Let $k\geq 0$ be an integer. Then:
           $$\sin(2kt) = \cos(t)\sum_{s=1}^{k} (-1)^{s-1}\cdot 2^{2s-1}\binom{k+s-1}{2s-1}
					    \cdot \sin^{2s-1}(t),\text{ for all }t\in\mathbb R.$$
\end{enumerate}
\end{prop}

\begin{proof}
(1) Since $\sin(t)=\cos(t-\frac{\pi}{2})$, the result follows from Proposition \ref{p1} and the identity:
$$\cos((2k+1)(t-\frac{\pi}{2})) = \cos((2k+1)t-k\pi-\frac{\pi}{2})=(-1)^k\cos((2k+1)t-\frac{\pi}{2})=(-1)^k\sin((2k+1)t).$$
(2) From Moivre's identity (\eqref{ec1}) it follows that 
\begin{equation}\label{ecc2}
\sin(2kt)=\sum_{j=1}^{k} (-1)^{j-1} \binom{2k}{2j-1} \cos^{2k+1-2j}(t)\sin^{2j-1}(t).
\end{equation}
Since $\cos^{2(k-j)}(t)=(1-\sin^2(t))^{k-j}$, from \eqref{ecc2} we deduce that
\begin{equation}\label{ecc3}
\sin(2kt) =\cos(t) \sum_{j=1}^{k} (-1)^{j-1} \binom{2k}{2j-1} \sum_{\ell=0}^{k-j} (-1)^{\ell} \binom{k-j}{\ell} \sin^{2j+2\ell-1}(t) = \\
\end{equation}
Denoting $s=j+\ell$, from \eqref{ecc3} it follows that
$$\sin(2kt)=\cos(t) \sum_{s=1}^k (-1)^{s-1} \sin^{2s-1}(t) \sum_{j=1}^{s} \binom{2k}{2j-1}\binom{k-j}{s-j}.$$
Hence, the required conclusion follows from Lemma \ref{lem2}.
\end{proof}

\begin{exm}
(1) According to Proposition \ref{p2}(1), we have 
$$\sin(3t)=-4\binom{3}{0}\sin^3(t)+\left(\binom{2}{1}+\binom{1}{0}\right)\sin(t)=3\sin(t)-4\sin^3(t).$$
(2) According to Proposition \ref{p2}(2), we have 
$$\sin(4t)=\cos(t)\left(2\binom{2}{1}\sin(t)-8\binom{3}{3}\sin^3(t)\right)=-4\cos(t)\sin(t)+8\cos(t)\sin^3(t).$$
\end{exm}

\begin{teor}\label{t1}
\begin{enumerate}
\item[(1)] Let $k\geq 0$ be an integer. Then:
           $$\cos^{2k}(t) = \frac{1}{2^{2k-1}} \sum_{j=0}^{k-1} \binom{2k}{j} \cos((2k-2j)t) + \frac{1}{2^{2k}}\binom{2k}{k}.$$
\item[(2)] Let $k\geq 0$ be an integer. Then:
           $$\cos^{2k+1}(t) = \frac{1}{2^{2k}} \sum_{j=0}^{k} \binom{2k+1}{j} \cos((2k+1-2j)t).$$
\end{enumerate}
\end{teor}

\begin{proof}
(1) We write $\cos^{2k}(t)=\frac{\beta_0}{2}+\sum_{j=1}^k \beta_j \cos(2jt)$. From Proposition \ref{p1}, it follows that
\begin{equation}\label{fun}
\cos^{2k}(t) = \frac{\beta_0}{2} + \sum_{j=1}^k \beta_j \sum_{s=0}^j (-1)^s 2^{2j-2s-1} \left( \binom{2j-s}{s}+\binom{2j-s-1}{s-1} \right) \cos^{2j-2s}(t).
\end{equation}
From \eqref{fun}, it follows that $1=2^{2k-1}\beta_k$ and thus $\beta_k=\frac{1}{2^{2k-1}}$. In order to prove (1),
we have to show that
\begin{equation}\label{funy}
 \beta_{\ell} = \frac{1}{2^{2k-1}}\binom{2k}{k-\ell},\text{ for all }0\leq \ell\leq k.
\end{equation}
We proceed by descending induction on $\ell$. The assertion is true for $\ell=k$. Suppose $\ell<k$. From \eqref{fun} it follows that
$$0 = \sum_{s=0}^{k-\ell} (-1)^s 2^{2\ell-1}\left( \binom{2\ell+s}{s} + \binom{2\ell+s-1}{s-1}\right) \beta_{\ell+s}.$$
Therefore, using the induction hypothesis, it follows that:
\begin{align*}
\beta_{\ell} &= -\sum_{s=1}^{k-\ell} (-1)^s \left( \binom{2\ell+s}{s} + \binom{2\ell+s-1}{s-1}\right) \beta_{\ell+s} = \\
 &= -\frac{1}{2^{2k-1}} \sum_{s=1}^{k-\ell} (-1)^s \left( \binom{2\ell+s}{s} + \binom{2\ell+s-1}{s-1}\right) \binom{2k}{k-\ell-s}.
\end{align*}
Hence, in order to prove \eqref{funy}, it is enough to show that
$$ \sum_{s=1}^{k-\ell} (-1)^s \left( \binom{2\ell+s}{s} + \binom{2\ell+s-1}{s-1}\right) \binom{2k}{k-\ell-s} = -\binom{2k}{k-\ell},$$
which is equivalent to
$$ \sum_{s=0}^{k-\ell} (-1)^s \left( \binom{2\ell+s}{s} + \binom{2\ell+s-1}{s-1}\right) \binom{2k}{k-\ell-s} = 0.$$
The last identity follows from Lemma \ref{cooc}, using the substition $t=k-\ell$.

(2) The proof is similar.
\end{proof}

\begin{exm}\rm
(1) Let $k=2$. According to Theorem \ref{t1}(1), we have:
    $$\cos^4(t)= \frac{1}{2^3}\left( \binom{4}{0}\cos(4t) + \binom{4}{1}\cos(2t) \right) + \frac{1}{2^4}\binom{4}{2} =\frac{1}{8} (\cos(4t)+4\cos(2t)+3).$$
(2) Let $k=1$. According to Theorem \ref{t1}(2), we have:
    $$\cos^3(t) = \frac{1}{2^2}\left( \binom{3}{0}\cos(3t)+\binom{3}{1}\cos(t) \right) = \frac{1}{4}(\cos(3t)+3\cos(t)).$$
\end{exm}

\begin{cor}\label{c1}
\begin{enumerate}
\item[(1)] Let $k\geq 0$ be an integer. Then:
           $$\sin^{2k}(t) = \frac{1}{2^{2k-1}} \sum_{j=0}^{k-1} (-1)^{k-j} \binom{2k}{j} \cos((2k-2j)t) + \frac{1}{2^{2k}}\binom{2k}{k}.$$
\item[(2)] Let $k\geq 0$ be an integer. Then:
           $$\sin^{2k+1}(t) = \frac{1}{2^{2k}} \sum_{j=0}^{k} (-1)^{k-j} \binom{2k+1}{j} \sin((2k+1-2j)t).$$
\end{enumerate}
\end{cor}

\begin{proof}
Since $\sin(t) = \cos(t-\frac{\pi}{2})$ and $\cos(t+n\pi)=(-1)^n\cos(t)$, from Theorem \ref{t1} it follows that
$$ \sin^{2k}(t) = \frac{1}{2^{2k-1}} \sum_{j=0}^{k-1} (-1)^{k-j}\binom{2k}{j}\cos((2k-2j)t) + \frac{1}{2^{2k}}\binom{2k}{k}.$$
Hence (1) holds. The proof of (2) is similar.
\end{proof}

\section{An application}

Let $a\in \mathbb R$ with $|a|>1$. We consider the function
$$f(t):=\frac{1}{a-\cos t},\text{ where }t\in\mathbb R.$$
Since $|\cos t|\leq 1$, we have the expansion
\begin{equation}\label{eco1}
f(t)=\sum_{n=0}^{\infty} a^{-n-1}\cos^n t,\text{ for all }t\in \mathbb R.
\end{equation}
From \eqref{eco1}, using Theorem \ref{t1}, we can write
\begin{equation}\label{eco2}
f(t)=\frac{a_0}{2}+\sum_{n=1}^{\infty}a_n\cos(nt),\text{ where }a_n=\frac{2}{a}\sum_{k\geq 0}(2a)^{-n-2k}\binom{n+2k}{k},\text{ for all }n\geq 0.
\end{equation}
Using the binomial expansion, we have
\begin{equation}\label{eco3}
(1-a^{-2})^{-\frac{1}{2}} = \sum_{k\geq 0}(-1)^n \binom{-\frac{1}{2}}{k} a^{-2n} = \sum_{k\geq 0} \frac{1\cdot 3\cdot 5\cdots (2k-1)}{2^k\cdot k!}a^{-2k}
= \sum_{k\geq 0} \binom{2k}{k} (2a)^{-2k}.
\end{equation}
From \eqref{eco2} and \eqref{eco3} it follows that
\begin{equation}\label{eco4}
a_0=\frac{2}{a\sqrt{1-a^{-2}}}.
\end{equation}

\begin{lema}\label{lemas}
With the above notations, we have that $(a_{n-1}-aa_n)a_0=2a_n$.
\end{lema}

\begin{proof}
From \eqref{eco1}, the required formula is equivalent to
$$ \sum_{k\geq 0} (2a)^{-n-2k}\left(2\binom{n+2k-1}{k}-\binom{n+2k}{k}\right)\cdot \sum_{t\geq 0}(2a)^{-2t}\binom{2t}{t} =
   \sum_{\ell=0}^{\infty} (2a)^{-n-2\ell}\binom{n+2\ell}{\ell}.$$
Since $2\binom{n+2k-1}{k}-\binom{n+2k}{k}=\binom{n+2k-1}{k}-\binom{n+2k-1}{k-1}$, the above identity is equivalent to
$$	\sum_{k\geq 0} (2a)^{-2k}\left(\binom{n+2k-1}{k}-\binom{n+2k-1}{k-1}\right)\cdot \sum_{t\geq 0}(2a)^{-2t}\binom{2t}{t} =
    \sum_{\ell=0}^{\infty} (2a)^{-2\ell}\binom{n+2\ell}{\ell}.$$
The last identity follows immediately from Lemma \ref{cheie}.		
\end{proof}

\begin{teor}\label{t2}
With the above notations, we have that:
$$f(t)=\frac{1}{a-\cos t} = \frac{1}{a\sqrt{1-a^{-2}}}+\sum_{n=1}^{\infty} \frac{2a^{n-1} (1-\sqrt{1-a^{-2}})^n}{\sqrt{1-a^{-2}}}\cos(nt).$$
\end{teor}

\begin{proof}
Note that $f(t)$ has expansion given in \eqref{eco2}, that is
$$\frac{1}{a-\cos t}=\frac{a_0}{2}+\sum_{n\geq 1}a_n\cos(nt).$$
Thus, in order to prove the theorem, we have to show that
 \begin{equation}\label{an}
 a_n=\frac{2a^{n-1}(1-\sqrt{1-a^{-2}})^{n}}{\sqrt{1-a^{-2}}},\text{ for all }n\geq 0.
\end{equation}
We use induction on $n\geq 0$. The case $n=0$ follows from \eqref{eco4}. In order to prove the induction step,
it is enough show that the desired expression of $a_n$ satisfies the recurrence relation given in Lemma \ref{lemas}.
Indeed, if we replace $a_n, a_{n-1}$ and $a_0$ with their required expressions, we get
\begin{equation}\label{opa1}
(a_{n-1}-aa_n)a_0 = \frac{ 2a^{n-2}(1-\sqrt{1-a^{-2}})^{n-1} - 2a^n (1-\sqrt{1-a^{-2}})^{n}  }{\sqrt{1-a^{-2}}} \cdot \frac{2}{a\sqrt{1-a^{-2}}}.
\end{equation}
On the other hand, we have:
\begin{equation}\label{opa2}
 \frac{ \frac{1}{a (1-\sqrt{1-a^{-2}} ) } - a }{a\sqrt{1-a^{-2}}} = \frac{ a(1+\sqrt{1-a^{-2}}) - a }{a\sqrt{1-a^{-2}}} = 1.
\end{equation}
From \eqref{opa1} and \eqref{opa2} it follows that $(a_{n-1}-aa_n)a_0=2a_n$, as required. Hence, the proof is complete.
\end{proof}

\begin{cor}\label{c2}
With the above notations, we have that:
\begin{align*}
 g(t)=\frac{1}{a-\sin t} = & \frac{1}{a\sqrt{1-a^{-2}}} + \sum_{k=1}^{\infty} (-1)^k\frac{2a^{2k-1}(1-\sqrt{1-a^{-2}})^{2k}}{\sqrt{1-a^{-2}}}\cos(2kt) + \\
                            & + \sum_{k=0}^{\infty} (-1)^k\frac{2a^{2k}(1-\sqrt{1-a^{-2}})^{2k+1}}{\sqrt{1-a^{-2}}}\sin((2k+1)t).
\end{align*}														
\end{cor}

\begin{proof}
Note that $\sin(t)=\cos(t-\frac{\pi}{2})$ implies $g(t)=f(t-\frac{\pi}{2})$. Hence,
the required formula follows from Theorem \ref{t2} and the identities:
$$\cos(nt - \frac{n\pi}{2}) = \begin{cases} (-1)^k \cos(nt),& n=2k \\ (-1)^k \sin(nt),& n=2k+1 \end{cases}.$$
\end{proof}

\begin{exm}\rm
(1) According to Theorem \ref{t2} we have that
$$f(t)=\frac{1}{\sqrt 3 - \cos t} = \frac{1}{2\sqrt 3} + \frac{1}{\sqrt 3} \sum_{n=1}^{\infty} (-\sqrt 3)^{n} \cos(nt),\text{ for all }t\in \mathbb R.$$
(2) According to Corollary \ref{c2} we have that
$$g(t)=\frac{1}{\sqrt 3 - \sin t}= \frac{1}{2\sqrt 3} + \frac{1}{\sqrt 3} \sum_{k=1}^{\infty} (-3)^k \cos(2kt) - \sum_{k=0}^{\infty} (-3)^k \sin((2k+1)t)
,\text{ for all }t\in \mathbb R.$$
\end{exm}

\end{document}